\documentclass[12pt]{article}

\usepackage[utf8]{inputenc}

\usepackage[a4paper,nomarginpar]{geometry}
\usepackage[english]{babel}

\usepackage{amsmath,amsfonts,amssymb,amsthm}
\usepackage{amscd}
\usepackage{graphics}

\usepackage{graphicx}
\usepackage{epsfig}
\usepackage{tikz}
\usepackage{multicol}
\usepackage{verbatim}
\usetikzlibrary{arrows}
\usetikzlibrary{shapes,snakes}

\theoremstyle{definition}
\newtheorem{theorem}{Theorem}
\newtheorem{lemma}[theorem]{Lemma}

\theoremstyle{definition}
\newtheorem{definition}[theorem]{Definition}

\newtheorem{remark}[theorem]{Remark}
\newtheorem{remarks}[theorem]{Remarks}
\newtheorem{problem}[theorem]{Problem}

\newcommand{\N}{\mathbb{N}} 
\newcommand{\R}{\mathbb{R}} 
\newcommand{\C}{\mathbb{C}} 
\newcommand{\ve}{\varepsilon}
\newcommand{\al}{\alpha}

\newcommand{\norm}[1]{\lVert#1\rVert}

\newcommand{\ldens}{\operatorname{\underline{dens}}}
\newcommand{\udens}{\operatorname{\overline{dens}}}

\begin{document}

\author{Luis Bernal-Gonz\'alez and Antonio Bonilla}

\title{Order of growth of distributional irregular entire functions for the differentiation operator}

\maketitle

\begin{abstract}
\noindent We study the rate of growth of entire functions that are distributionally irre\-gu\-lar for the differentiation operator $D$. More specifically, given \,$p \in [1,\infty ]$ \,and \,$b \in (0,a)$, where \,$a = \displaystyle \frac{1}{2\, \max\{2,p\}}$, we prove that there exists a distributionally irregular entire function \,$f$ \,for the operator \,$D$ \,such that its $p$-integral mean function \,$M_p(f,r)$ \,grows not more rapidly than \,$e^r r^{-b}$. This completes related known results about the possible rates of growth of such means for $D$-hypercyclic entire functions. It is also obtained the existence of dense li\-ne\-ar submanifolds of \,$H(\C )$ \,all whose nonzero vectors are $D$-distributionally irregular and present the same kind of growth.
\end{abstract}

\section{Introduction}

\quad In 1988, Beauzamy \cite{beau}, when trying to describe the erratic dynamics of certain vectors under the action of concrete operators, introduced the following notion.

\begin{definition}
Let $X$ be a Banach space and $T : X \rightarrow X$ be a bounded operator. A vector $x_0\in X$ is said to be \emph{irregular} for $T$ if
$\lim \inf_{n \to \infty} \norm{T^nx} =0$ and $\lim \sup_{n \to \infty} \norm{T^nx}=\infty$.
\end{definition}

Inspired by this definition, and by the notion of distributional chaotic mapping due to Schweizer and Sm\'{\i}tal (\cite{schweizersmital}, see also \cite{oprocha}), in \cite{BBMP11} it is considered the stronger property contained in Definition \ref{def-distributionally irregular} below. Recall that, if $A$ is a subset of the set $\N$ of positive integers, then its upper density and its lower density are respectively defined by
$$
\udens(A) = \limsup_{n \to \infty} {{\rm card} (A \cap \{1,2, \dots ,n\}) \over n}, \,\,
\ldens(A) = \liminf_{n \to \infty} {{\rm card} (A \cap \{1,2, \dots ,n\}) \over n}.
$$

\begin{definition} \label{def-distributionally irregular}
Let $X$ be a Banach space and $T : X \rightarrow X$ be a bounded operator.
A vector \,$x_0 \in X$ \,is said to be \emph{distributionally irregular} for \,$T$ \,if there are increasing sequences of positive
integers \,$A = (n_k)_k$ \,and \,$B=(m_k)_k$ \,such that
$$
\udens(A) = 1 = \udens(B), \ \ \lim_{k \to \infty} \norm{T^{n_k}x} = 0 \hbox{ \ and \ } \lim_{k \to \infty} \norm{T^{m_k} x_0}=\infty .
$$
\end{definition}

Now, let \,$Y$ \,be a Fr\'echet space, that is, $Y$ is a vector space endowed with an increasing sequence
$(\|\cdot\|_k)_{k \in \N}$ of seminorms (called a {\em fundamental sequence of seminorms}) that defines a metric
\[
 d(x,y) := \sum_{k=1}^\infty \frac{1}{2^k} \min\{1,\norm{x-y}_k\} \quad (x,y \in Y),
\]
under which \,$Y$ \, is complete. By \,$B(Y)$ \,we have denoted, as usual, the set of all continuous linear operators \,$T : Y \to Y$.

\vskip .15cm

The following concepts, that are a generalization of the above ones to Fr\'{e}chet spaces, were introduced in \cite{BBMP}.

\begin{definition}
Given \,$T \in B(Y)$ \,and \,$x_0 \in Y$, we say that \,$x_0$ \,is an \emph{irregular vector} for \,$T$ \,if there are \,$m \in \N$ and strictly increasing sequences $(n_k)$ and $(j_k)$ of positive integers such that
\[
 \lim_{k \to \infty} T^{n_k} x_0 = 0
 \ \ \mbox{ and } \ \
 \lim_{k \to \infty} \norm{{T^{j_k} x_0}}_m = \infty.
\]
\end{definition}

\begin{definition}
Given \,$T \in B(Y)$ \,and \,$x_0 \in Y$, we say that \,$x_0$ \,is a \emph{distributionally irregular vector} for \,$T$ \,if there are \,$m \in \N$ \,and \,$A,B \subset \N$ \,with \,$\udens(A) = 1 = \udens(B)$ \,such that
\[
 \lim_{n \to \infty \atop n \in A} T^{n} x_0 = 0
 \ \ \mbox{ and } \ \
 \lim_{n \to \infty \atop n \in B} \norm{{T^{n} x_0}}_m = \infty.
\]
\end{definition}

Let us consider the Fr\'echet space \,$H(\Bbb C)$ \,of entire functions endowed of the family of seminorms \,$\|f\|_{\infty,\overline{B(0,k)}}$
\,$(k \in \N )$. Here $B(a,r)$ denotes, as usual, the open ball in the complex plane \,$\C$ \,with center \,$a$ \,and radius \,$r > 0$; and, for a nonempty set $A \subset \C$, we have set \,$\|f\|_{\infty,A} := \sup \{|f(z)|: \, z \in A\}$. Hence the metric
\, $d(f,g) = \displaystyle \sum _{k=1}^{\infty} \frac{1}{2^k} \min\{1,\|f\|_{\infty,\overline{B(0,k)}}\}$ \,defines the topology of \,$H(\C )$.
If \,$D : f \in H(\C ) \mapsto f' \in H(\C )$ \,is the differentiation operator, then the above definitions read as follows.

\begin{definition}
Given $f \in H(\Bbb C)$, we say that $f$ is an \emph{irregular function} for \,$D$ \,if there are $m \in \N$ and increasing sequences
$(n_k)$ and $(j_k)$ of positive integers such that
\[
 \lim_{k \to \infty} D^{n_{k}}f = 0
 \ \ \mbox{ and } \ \
 \lim_{k \to \infty} \|D^{j_k}f\|_{\infty, \overline {B(0,m)}} = \infty.
\]
\end{definition}

\begin{definition}\label{def-D-distrirregular}
Given $f \in H(\Bbb C)$, we say that $f$ is a \emph{distributionally irregular function} for \,$D$ \,if there are $m \in \N$ and $A,B \subset \N$ with
$\udens(A) = \udens(B) = 1$ such that
\[
 \lim_{n \to \infty \atop n \in A} D^{n}f = 0
 \ \ \mbox{ and } \ \
 \lim_{n \to \infty \atop n \in B} \|D^{n}f\|_{\infty, \overline {B(0,m)}} = \infty.
\]
\end{definition}

In \cite{BBMP} is proved that $D$ admits distributionally irregular entire functions. In fact, it is shown in \cite{BBMP} that there are $T$-distributionally irregular entire functions for any given non-scalar operator \,$T$ \,on \,$H(\C )$ \,that commutes with \,$D$.

\vskip .15cm

Another important property related to the dynamics of operators is that of hypercyclicity. If $T$ is an operator defined on a topological vector space $X$ and $x_0 \in X$, then the vector $x_0$ is called hypercyclic for $T$ provided that the orbit \,$\{T^n x_0: \, n \in \N\}$ \,is dense in $X$; and, according to Bayart and Grivaux \cite{BayGri}, $x_0$ is called frequently hypercyclic for $T$ whenever the following stronger property is satisfied: for any prescribed nonempty open set $U \subset X$, $\ldens (\{n \in \N : \, U \cap T^n x_0 \ne \varnothing \}) > 0$. The existence of entire functions which are hypercyclic (frequently hypercyclic) for \,$D$ \,is known since MacLane \cite{MacLane} (Bayart and Grivaux \cite{BayGri}, resp.). For excellent accounts of these topics, we refer the reader to the books \cite{BaM,GEP}. Notice that, if $X$ is a Fr\'echet space, every
$T$-hypercyclic vector is $T$-irregular.

\vskip .15cm

Many papers have been devoted to study the {\it rate of growth} of entire functions that are hypercyclic or frequently hypercyclic for \,$D$, see for instance \cite{BeBo02,BlaBoGE09,Bo09,BB,BoGE06,BoGE07,DS,Gro90,Shk93}. In this paper, we aim to study the rate of growth of entire functions that are {\it irregular} or {\it distributionally irregular} for the differentiation operator. The growth for the ``irregular'' case will be completely determined, while lower and upper bounds will be provided for the growth of $D$-distributionally irregular functions. This research is completed by analyzing the distributional irregularity in weighted Banach spaces as well as the existence of large vector subspaces of $D$-distributionally irregular functions satisfying the mentioned growth conditions.

\section{Order of growth of distributional irre\-gu\-lar entire functions}

\quad For every $r > 0$, every entire function $f$ \,and \,$1\leq p < \infty$ \,we will consider the integral $p$-means
\[
M_p(f,r)=\Big( \frac{1}{2\pi}\int_0^{2\pi} |f(re^{it})|^p \,dt \Big)^{1/p}
\]
as well as the function \,$M_\infty(f,r)=\sup_{|z|=r} |f(z)|$.

\vskip .15cm

In 1990, Grosse-Erdmann \cite{Gro90} discovered that there are not $D$-hypercyclic entire functions $f$ with $M_\infty (f,r)$ growing not more rapidly than $e^r/\sqrt{r}$, while there do exist $D$-hypercyclic entire functions growing under any prescribed rate speeder than $e^r/\sqrt{r}$.
Our first result states that the critical order of growth for hypercyclic and irregular entire functions for the differentiation operator
is the same, and that this is independent of the $p$-mean considered.

\begin{theorem}
\textit{Let \,$1\leq p\leq \infty$. We have:
\begin{enumerate}
\item[\rm (a)] For any function \,$\varphi:\mathbb{R}_+\to\mathbb{R}_+$
\,with \,$\varphi(r)\to\infty$ \,as $r\to\infty$ \,there is a $D$-irregular entire function \,$f$ \,with
\[
M_p(f,r)\leq \varphi(r)\frac{e^r}{\sqrt{r}}\quad\mbox{for $r>0$
sufficiently large}.
\]
\item[\rm(b)] There is no $D$-irregular entire function $f$ that satisfies
\[
M_p(f,r)\leq C \frac{e^r}{\sqrt{r}}\quad \mbox{for } r > 0,
\]
where \,$C$ \,is a positive constant.
\end{enumerate}}
\end{theorem}

\begin{proof}
\noindent Part (a) follows from the corresponding result for $D$-hypercyclic functions \cite{Gro90}
(see also \cite{Shk93}), since all hypercyclic vectors are irregular.

\vskip .15cm

As for (b), assume that $f$ is an entire function satisfying $M_p (f,r) \le C \frac{e^r}{\sqrt{r}}$ $(r > 0)$ for some $C > 0$.
Fix $m \in \N$ and a radius $R > m$. From Cauchy's integral formula for derivatives of holomorphic functions, we get for all $n \in \N$ that
$$
D^{n}f(z) = {n! \over 2 \pi i} \oint_{|\xi | = R} {f(\xi ) \over (\xi - z)^{n+1}} \, d\xi \quad (z \in  \overline{B(0,r)}).
$$
>From here and the triangle inequality one derives that
\[
\|D^{n}f\|_{\infty,\overline{B(0,m)}} \leq \frac{n!\,  R}{(R-m)^{n+1}} M_1(f,R).
\]
Since $M_1(f,R)\leq M_p(f,R) \le C \, e^R / \sqrt{R}$, we find that \,$\|D^n f\|_{\infty, \overline{B(0,m)}} \leq
\frac{C \, n! \,e^R \, R^{1/2}}{(R - m)^{n+1}}$ \,for all $n \in \N$. Letting \,$R = n$ \,we get
$$
\|D^n f\|_{\infty, \overline{B(0,m)}} \leq { C \, n! \; e^{n}\over n^{n + {1 \over 2}}  (1 - {m \over n})^{n+1}}
$$
for all $n > m$. Now Stirling's formula \,$n! \sim \sqrt{2 \pi n} \,e^{-n}\, n^n$ ($n \to \infty$)
\,and the fact \,$(1-{m\over n})^{n+1} \to e^{-m}$ $(n \to \infty )$ \,imply
that the sequence $\{\|D^n f\|_{\infty, \overline{B(0,m)}}\}_{n \ge 1}$ \,is bounded, so that \,$f$ \,cannot be irregular.
\end{proof}

The following auxiliary result, which can be found in \cite[Lemma 2.2]{BlaBoGE09}, will be needed in the sequel.

\begin{lemma}\label{L-Bar}
{\it Let \,$0<\alpha\leq 2$ \,and \,$\beta\in \mathbb{R}$. Then there is a constant \,$C>0$ \,such that
\[
\sum_{n=0}^\infty\frac{r^{\alpha n}}{(n+1)^\beta (n!)^\alpha}\leq
C r^{\frac {1-\alpha-2\beta}{2}}e^{\alpha r} \hbox{ \ for all } r > 0.
\]}
\end{lemma}

Let $2\leq p\leq\infty$, by the Hausdorff--Young inequality (see, for example, \cite{Ka76}) we obtain that
\[
M_p\Big(\sum a_n\frac{z^n}{n!},r\Big) \le \Big( \sum |a_n|^q\frac{r^{qn}}{(n!)^{q}}\Big)^{1/q},
\]
where $q := {p \over p-1}$ is the conjugate exponent of $p$.

\vskip .2cm

Our second result provides the construction of a $D$-distributionally irregular entire function having, in some sense, prescribed control on their
Taylor coefficients.

\begin{theorem} \label{coefficientstheorem}
{\it Assume that \,$\{\omega_n \}_{n \ge 1} \subset [0,+\infty )$ \,is a sequence with
$\lim_{n \to \infty} \omega_n = + \infty$.
Then there exists an entire function \,$f$ \,satisfying the following properties:
\begin{enumerate}
\item[\rm (a)] $|f^{n}(0)| \le \omega_n$ \,for all \,$n \ge 1$, and
\item[\rm (b)] $f$ \,is $D$-distributionally irregular in \,$H(\C )$.
\end{enumerate}
}
\end{theorem}

\begin{proof}
Firstly, the radius of convergence of the series \,$\sum_{n=1}^\infty n^{1 + {n \over 2}} \, {z^n \over n!}$ \,is
$$R = \lim_{n \to \infty} {n^{1 + {n \over 2}}/n! \over (n+1)^{1 + {n+1 \over 2}}/(n+1)!} = +\infty .$$
Therefore it
converges at every \,$z \in \C$. In particular, we have
$$\lim_{k \to \infty} \sum_{n=k}^\infty n^{1 + {n \over 2}} \, {R^n \over n!} = 0 \hbox{ \ for any } \,R > 0. \eqno (1)$$

Let us introduce some notation. Define \,$\widetilde{\omega_n} := \min \{\omega_n,n\}$ $(n \ge 1)$. Since
\,${\widetilde{\omega_n} \over n!} \le {1 \over (n-1)!}$ for all $n \in \N$, we have obtain that,
for any selection of the set \,$S \subset \N$,
the expression \,$\sum_{n \in S} \widetilde{\omega_n} \, {z^n \over n!}$ \,defines an entire function.

\vskip .15cm

The core of the proof is the search for an entire function \,$f$ \,having large derivatives \,$D^{j}f$ \,for many indexes $j$ and, simultaneously,
having small derivatives \,$D^{j}f$ \,for other (many) indexes $j$.

\vskip .15cm

Thanks to (1), we can select \,$\al_1 \in \N$ \,such that
$$\sum_{n=\al_1}^\infty n^{1 + {n \over 2}} \, {1 \over n!} < 1.$$
Now, we choose \,$\beta_1 \in \N$ \,with \,$\beta_1 > 2 \alpha_1^2$. >From (1), again,
we obtain an \,$\alpha_2 \in \N$ \,with \,$\alpha_2 > \beta_1^2$ \,such that
$$\sum_{n=\al_2}^\infty n^{1 + {n \over 2}} \, {2^n \over n!} < {1 \over 2}.$$
Choose \,$\beta_2 \in \N$ \,with \,$\beta_2 > 2 \alpha_2^2$. Proceeding in this way, we can recursively obtain
two sequences \,$\{\al \}_{n \ge 1}$ \,and \,$\{\beta \}_{n \ge 1}$ \,of natural numbers satisfying
$$
\al_1 < 2 \al_1^2 < \beta_1 < \beta_1^2 < \cdots < \alpha_n < 2 \al_n^2 < \beta_n < \beta_n^2 < \alpha_{n+1} < \cdots
$$
and
$$
\sum_{n=\al_N}^\infty n^{1 + {n \over 2}} \, {N^n \over n!} < {1 \over N} \hbox{ \ for all } N \in \N . \eqno (2)
$$

Now, define the sets
$$
A := \bigcup_{n=1}^\infty \{\al_n,\al_n + 1, \dots ,\al_n^2\} \hbox{ \ and \ } B := \bigcup_{n=1}^\infty \{\beta_n,\beta_n + 1, \dots ,\beta_n^2\}.
$$
Notice that
\begin{align*}
\overline{{\rm dens}} \, (A) & = \limsup_{n \to \infty}  {{\rm card} (A \cap \{1,\dots ,n\}) \over n} \\
& \ge  \limsup_{n \to \infty}  {{\rm card} (A \cap \{1,\dots ,\al_n^2\}) \over \al_n^2}\\
& =  \limsup_{n \to \infty}  {{\rm card} (\{\al_n , \al_n+1,\dots ,\al_n^2\}) \over \al_n^2}\\
& = \lim_{n \to \infty}  {\al_n^2 - \al_n + 1 \over \al_n^2} = 1.
\end{align*}
Hence \,$\overline{{\rm dens}} \, (A) = 1$ \,and, analogously, $\overline{{\rm dens}} \, (B) = 1$.

\vskip .15cm

Next, we set, for \,$n \in \N$:
$$
\omega_n^* := \left\{
\begin{array}{cl}
\widetilde{\omega_n} & \mbox{if } n \in B\\
0 & \mbox{otherwise}
\end{array}
\right.
$$
and define the entire function
$$f(z) = \sum_{n=1}^\infty \omega_n^* \, {z^n \over n!} = \sum_{n \in B} \widetilde{\omega_n} \, {z^n \over n!}.$$
Clearly, $|f^{(n)}(0)| = \omega_n^* \le \widetilde{\omega_n} \le \omega_n$ \,for all \,$n \ge 1$.
Therefore, our only task is to prove that \,$f$ \,is distributionally irregular for \,$D$ \,in \,$H(\C )$.

\vskip .15cm

With this aim, fix any \,$m \in \N$. If \,$n \in B$ \,then \,$\|D^{n}f\|_{\infty, \overline{B(0,m)}} \ge |f^{(n)}(0)| = \widetilde{\omega_n} = \min \{\omega_n ,n\}$, which entails that \,\,$\lim_{n \to \infty \atop n \in B} \|D^n f\|_{\infty, \overline {B(0,m)}} = \infty$.

\vskip .15cm

Finally, in order to show that \,$\lim_{n \to \infty \atop n \in A} D^{n}f = 0$ \,it is enough to prove that
\,$\lim_{n \to \infty \atop n \in A} \|f^{(n)}\|_{\infty, \overline {B(0,m)}} = 0$ \,for every \,$m \in \N$.
So, fix \,$m \in \N$ \,as well as an \,$\ve > 0$. Choose \,$N_0 \in \N$ \,with \,$N_0 \ge m$ \,and \,$1/N_0 < \ve$.
Denote \,$j_0 := \al_{N_0}$. For every \,$j \in A$ \,with \,$j \ge j_0$ \,there is (a unique) \,$N \ge N_0$
\,such that \,$\al_N \le j \le \al_N^2$. Then we obtain
\begin{equation*}
\begin{split}
f^{(j)} (z) &= \sum_{n=j}^\infty \omega_n^* \, n(n-1)(n-2) \cdots (n-j+1) \, {z^n \over n!} \\
            &= \sum_{n = 2 \al_N^2}^\infty \omega_n^* \, n(n-1)(n-2) \cdots (n-j+1) \, {z^n \over n!},
\end{split}
\end{equation*}
because \,$\omega_n^* = 0$ \,if \,$n \not\in B$. From (2), the triangle inequality and the
fact \,$\al^2_N > 2 \, \al_N$ \,we get for every \,$z \in \overline {B(0,m)}$ \,that
\begin{equation*}
\begin{split}
|f^{(j)} (z)| &\le \sum_{n = 2 \al_N^2}^\infty \omega_n^* \, n(n-1)(n-2) \cdots (n-j+1) \, {m^n \over n!} \\
              &\le \sum_{n = 2 \al_N^2}^\infty \omega_n^* \, n^j \, {m^n \over n!}
               \le \sum_{n = 2 \al_N^2}^\infty n \cdot n^{n \over 2} \, {N^n \over n!} \\
              &\le \sum_{n = \al_N}^\infty  n^{1 + {n \over 2}} \, {N^n \over n!} < {1 \over N} \le {1 \over N_0} < \ve .
\end{split}
\end{equation*}
Therefore \,$\|f^{(j)}\|_{\infty, \overline {B(0,m)}} < \ve$ \,whenever \,$j \in A$ \,and \,$j \ge j_0$.
In other words, $\lim_{j \to \infty \atop j \in A} \|f^{(j)}\|_{\infty, \overline {B(0,m)}} = 0$, and we are done.
\end{proof}

\begin{remarks}\label{comment}
{\rm
1. The idea of making large gaps in the sequence of Taylor coefficients, given in the previous proof,
is inspired by the construction of scrambled sets for weighted backward shifts on K\"othe sequences spaces due to Wu \cite{Wu1} and Wu {\it et al.}~\cite{Wu2}.
\vskip 3pt
\noindent 2. Observe that the function \,$f$ \,constructed in the proof of Theorem \ref{coefficientstheorem} is $D$-distributionally irregular in a sense stronger than the one given in Definition \ref{def-D-distrirregular}, because the set \,$B$ \,satisfying
\,$\lim_{n \to \infty \atop n \in B} \|D^n f\|_{\infty, \overline {B(0,m)}} = \infty$ \,holds {\it for any} $m \in \N$.
\vskip 3pt
\noindent 3. Observe that the sequences \,$\{\al_n\}_{n \ge 1}$ \,and \,$\{\beta _n\}_{n \ge 1}$ \,(hence the sets \,$A$ \,and \,$B$)
are {\it independent} of the sequence \,$\{\omega_n\}_{n \ge 1}$. This fact will be exploited in the next section.
}
\end{remarks}

Now, we are ready to present the main result of this section.

\begin{theorem} \label{main theorem}
{\it Let \,$1\le p\le \infty$, and set \,$a = \displaystyle \frac{1}{2 \max\{2,p\}}$. Then, for every \,$\ve > 0$, there exists a distributionally irregular entire function \,$f$ \,for the differentiation operator acting on \,$H(\C )$ \,such that
$$
M_{p}(f,r)\leq C\displaystyle \frac {e^r}{r^{a-\ve}}  \quad (r > 0)
$$
for some constant \,$C>0$.}
\end{theorem}

\begin{proof}
Since
\[
M_p(f,r)\leq M_{2}(f,r)\quad \mbox{for $1\leq p < 2$,}
\]
we need only prove the result for $p \geq 2$.

\vskip .15cm

Then fix \,$p \ge 2$ \,as well as an \,$\ve > 0$ \,and denote, as usual, by \,$q$ \,the conjugate exponent of \,$p$.
Take \,$\omega_n := n^\ve$ $(n \ge 1)$. Of course, we have \,$\omega_n \to +\infty$. By Theorem \ref{coefficientstheorem}
there exists a $D$-distributionally irregular entire function \,$f$ \,satisfying ($f(0)=0$ and)
\,$|f^{(n)}(0)| \le n^\ve$ \,for all $n \ge 1$.

\vskip .15cm

Finally, making use of the
Hausdorff-Young inequality  and Lemma \ref{L-Bar} (with $\alpha = q$ and $\beta = - \ve q$), we have that
$$
M_{p}(f,r) \leq \left(\sum_{n=1}^{\infty}  |f^{(n)}(0)| \, (\frac{r^{n}}{n!})^q \right)^{\frac{1}{q}}
\leq \left(\sum_{n=1}^{\infty}  n^{\ve q}(\frac{r^{n}}{n!})^q \right)^{\frac{1}{q}} \le C \, \frac {e^r}{r^{\frac{1}{2p}-\ve}}
= C \,\frac {e^r}{r^{a - \ve}}
$$
for all \,$r > 0$ \,and some positive constant \,$C$, which proves the theorem.
\end{proof}

The following figure  shows our present knowledge of possible or impossible rates $\frac {e^r}{r^{a}}$ for distributionally irregular entire
function for differentiation operator.

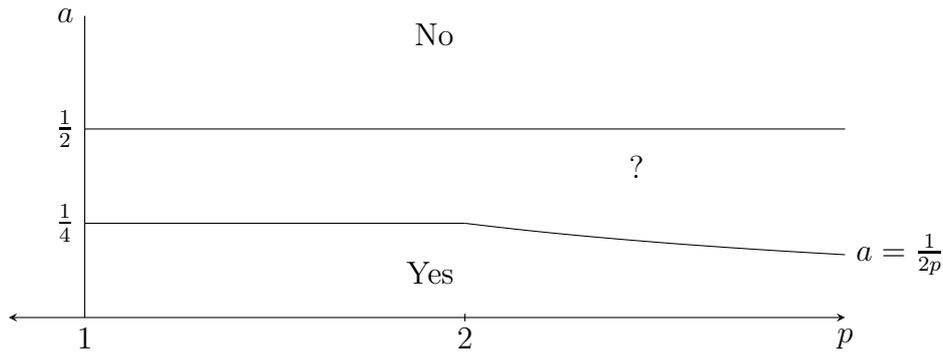
\begin{center}
\begin{tikzpicture}[scale=5,>=stealth]
 \draw  (1,0)--(1,0.8);
 \draw [<->] (0.8,0)--(3,0);
 \node[left] at (1,0.25) {$\frac{1}{4}$};
 \draw  (1,0.25)--(2,0.25);
 \node[left] at (1,0.5) {$\frac{1}{2}$};
 \node[below] at (1,0) {1};
 \node[below] at (2,0) {2};
   \draw (2,0.01)--(2,-0.01);
 \draw  (1,0.5)--(3,0.5);
 \draw  plot [domain=2:3] (\x,1/2*1/\x);
 \node[right] at (3,1/6) {$a=\frac{1}{2p} $};
 \node[below] at (3,0) {$p $};
 \node[left] at (1,0.8) {$a$};
 \node[left] at (2,0.75) {No};
  \node[left] at (2,.12) {Yes};
  \node[left] at (2.5,0.4) {?};
   \end{tikzpicture}
\end{center}

\begin{problem}
For each \,$p \in [1, \infty ]$, give the critical order of growth between $\frac{e^r}{\sqrt{r}}$
and $\frac {e^r}{r^{a-\ve}}$  \,(with  $a = \frac{1}{2 \max\{2,p\}}$)
\,of a distributionally irregular entire function for the differentiation operator.
\end{problem}

\section{Dense linear manifolds of distributionally irre\-gu\-lar functions}

\quad The study of lineability, that is, the search of linear (or, in general, algebraic) structures
within nonlinear sets has become a trend in the last two decades, see e.g.~the survey \cite{BPS}.
In particular, if \,$X$ \,is a topological vector space (over \,$\R$ \,or \,$\C$) and \,$S$ \,is a subset of \,$X$, then
\,$S$ \,is called {\it dense-lineable} in \,$X$ \,provided that there exists a dense vector subspace \,$M$ \,of \,$X$ \,such that
\,$M \subset S \cup \{0\}$. If, in addition, $M$ \,can be found with \,${\rm dim} \,(M) = {\rm dim} \, (X)$, then \,$S$
\,is said to be {\it maximal dense-lineable} in \,$X$. In this section, we consider the lineability of the family of
$D$-distributionally irregular entire functions having growth restrictions.

\vskip .15cm

To this end, we will need the next lemma, whose content can be found in \cite{BPS} (see also \cite{AGPS,Bernal2008,Bernal2010,BerOrd}).
Following \cite{AGPS}, if \,$\mathcal{A, \, B}$ \,are subsets of a vector space, then we say that
\,$\mathcal A$ \, is {\it stronger than} \,$\mathcal B$ \,whenever \,${\mathcal A} + {\mathcal B} \subset {\mathcal A}$.

\begin{lemma}\label{maximal dense-lineable}
{\it Assume that \,$X$ \,is a metrizable topological vector space. Let \,$\mathcal{A} \subset X$. Suppose that there exists a dense vector subspace       \,$\mathcal{B} \subset X$ \,such that \,$\mathcal{A}$ \,is stronger than \,$\mathcal{B}$ \,and \,$\mathcal{A} \, \cap \, \mathcal{B} = \emptyset$.
Suppose also that \,$\mathcal{A} \cup \{0\}$ \,contains a vector subspace whose dimension equals {\rm dim}$\,(X)$.  Then \,$\mathcal{A}$ \,is maximal dense-lineable.}
\end{lemma}

Denote by \,$\mathcal C$ \,the class of functions satisfying simultaneously all properties and growth restrictions considered in the previous section. More precisely, we denote
\begin{equation*}
\begin{split}
\mathcal{C} := \{f \in H(\C ): \,\, &f \hbox{ is } D\hbox{-distributionally irregular in } H(\C ) \hbox{ and } \\
                                    &\sup_{r > 0} r^{\frac{1}{2\, \max\{2,p\}} - \ve} \, e^{-r} \, M_p(f,r) < \infty \\
                                    &\hbox{ for all } \ve > 0 \hbox{ and all } p \in [1,\infty ]\}.
\end{split}
\end{equation*}

The following theorem reveals a rich linear structure inside this (seemingly small) class.

\begin{theorem}
{\it The set \,$\mathcal C$ \,is maximal dense-lineable in \,$H(\C )$.}
\end{theorem}

\begin{proof}
We apply Theorem \ref{coefficientstheorem} to the sequence \,$\omega_n := (\log (n+1))^t$, where \,$t > 0$.
By the construction given in the proof of Theorem \ref{coefficientstheorem} (whose notation we keep here) and Remark \ref{comment}.3, there are subsets \,$A$ \,and \,$B$ \,of \,$\N$ \,with maximal upper density satisfying that, for each \,$t > 0$, the entire function
$$
f_t(z) := \sum_{n \in B} \min \{n,(\log (n+1))^t\} \, {z^n \over n!}
$$
is $D$-distributionally irregular in \,$H(\C )$. In fact, on one hand, we have that \,$|f^{(n)}_t(0)| = \min \{n,(\log (n+1))^t\}$ \,for all
\,$n \in B$ \,and, on the other hand, the proof of Theorem \ref{coefficientstheorem} allows to obtain that \,$\lim_{n \to \infty \atop n \in A} \|D^n f_t\|_{\infty, \overline {B(0,m)}} = 0$ \,for all \,$m \in \N$ \,and all \,$t > 0$.

\vskip .15cm

Observe first that the functions \,$f_t$ $(t > 0)$ \,are linearly independent. Indeed, assume that \,$c_1, \dots ,c_s$ \,are complex numbers (with $s \ge 2$ \,and \,$c_s \ne 0$) and that \,$0 < t_1 < \cdots < t_s$. If the linear combination
$$
F := \sum_{k=1}^s c_k f_{t_k} \eqno (3)
$$
is identically zero then, after derivation, we get
$$
|F^{(n)}(0)| = \big|\sum_{k=1}^s c_k f^{(n)}_t (0) \big| = \big|\sum_{k=1}^s c_k f^{(n)}_t (0) \big| = \big|\sum_{k=1}^s c_k \, \min \{n,(\log (n+1))^{t_k}\} \big| \,\, (n \in B).
$$
Since \,$(\log (n+1))^{t_k} \le (\log (n+1))^{t_s}$ $(n \in \N ; \, k=1, \dots ,s)$ \,and \,${n \over (\log (n+1))^{t_s}} \to +\infty$ \,as
\,$n \to \infty$, one derives that \,$|F^{(n)}(0)| = |\sum_{k=1}^s c_k \, \log (n+1))^{t_k}|$ \,for $n \in B$ large enough. Hence
\,$|F^{(n)}(0)| \to +\infty$ \,as \,$n \to \infty$ $(n \in B)$, which is absurd. Then \,$F$ \,is not identically $0$, which shows the linear independence of \,$\{f_t: \, t > 0\}$.

\vskip .15cm

Define \,$M := {\rm span} \, \{f_t: \, t > 0\}$. According to the previous paragraph, $M$ \,is a vector subspace of \,$H(\C )$ \,with
\,${\rm dim} \, (M) = \mathfrak{c} = {\rm dim} \, (H(\C ))$, where \,$\mathfrak{c}$ \,denotes the cardinality of the continuum. Fix any
$F \in M \setminus \{0\}$. Then \,$F$ \,has the form (3), with the \,$c_k$'s and the \,$t_k$'s as above. Therefore, for given
\,$m \in \N$, we have \,$\lim_{n \to \infty \atop n \in B} \|D^{n}F\|_{\infty, \overline{B(0,m)}} \ge \lim_{n \to \infty \atop n \in B} |F^{(n)}(0)| = +\infty$. In addition, by the triangle inequality,
$$
\lim_{n \to \infty \atop n \in A} \|D^n F\|_{\infty, \overline {B(0,m)}} \le \sum_{k=1}^s |c_k| \lim_{n \to \infty \atop n \in A} \|D^n f_{t_k}\|_{\infty, \overline {B(0,m)}} = 0.
$$
This shows that \,$F$ \,is $D$-distributionally irregular.
Now, given \,$\ve > 0$, there is a constant \,$K = K(\ve ,c_1, \dots ,c_s,t_1, \dots ,t_s) \in (0,+\infty)$ \,such that
\,$|F^{(n)}(0)| \le K \, n^\ve$ \,for all \,$n \ge 1$. The  approach of the final part of the proof of Theorem \ref{main theorem} yields the existence of positive constants \,$C = C_{\ve ,p}$ $(p \in [1,\infty ])$ \,satisfying
$$
M_p(F,r) \le C_{\ve ,p} \, {e^r \over r^{\frac{1}{2\, \max\{2,p\}} - \ve}} \hbox{ \ for all } r > 0.
$$
In other words, $F \in {\mathcal C}$. Thus, our class \,${\mathcal C}$ \,contains, except for $0$, a vector space having maximal dimension.

\vskip .15cm

Finally, let \,$X := H(\C )$, ${\mathcal A} := {\mathcal C}$ \,and \,${\mathcal B} := \{$complex polynomials$\}$.
Recall that \,${\mathcal B}$
\,is a dense vector subspace of \,$H(\C )$. Moreover, given \,$P \in {\mathcal B}$, one has \,$P^{(n)} = 0$ \,for \,$n$ \,large enough. It follows, trivially, that \,$P$ \,is not $D$-distributionally irregular but \,$P + f$ \,is $D$-distributionally irregular if \,$f$ \,is. In addition, it is an easy exercise to prove that, for every \,$b \in \R$ and every \,$p \in [1,\infty ]$, the set \,$\{f \in H(\C ): \, \sup_{r > 0} r^b \, e^{-r} \, M_p(f,r) < \infty\}$ \,is a vector space containing the polynomials. Consequently, ${\mathcal A} \, \cap {\mathcal B} = \varnothing$ \,and \,${\mathcal A} + {\mathcal B} \subset \mathcal{A}$. An application of Lemma \ref{maximal dense-lineable} yields the maximal dense-lineability of \,${\mathcal C}$.
\end{proof}

\section{Weighted Banach spaces of entire functions}

\quad In this brief section, we establish the existence of distributionally irregular vectors for the
differentiation operator acting on certain weighted Banach spaces of entire functions. As a sub-product,
large linear manifolds consisting of such vectors will be obtained again.

\vskip .15cm

A {\it weight} \,$v$ \,on \,$\C$ \,will be is a strictly positive continuous
function on $\C$ which is radial, i.e.~$v(z)=v(|z|)$ $(z \in \C )$,
such that \,$v(r)$ \,is non-increasing on \,$[0,\infty )$ \,and satisfies
\,$\lim_{r \rightarrow \infty} r^m v(r)=0$ \,for each \,$m \in \N$.

\vskip .15cm

We define, for $1 \le p \le \infty$ and a weight function $v$, the following spaces as in \cite{lusky2000on}:
$$
B_{p,\infty}=B_{p,\infty}(\C,v):= \{ f\in H(\Bbb C) : \, \sup_{r>0} v(r)M_{p}(f,r) <\infty \}
$$
and
$$
B_{p,0}=B_{p,0}(\C,v):= \{ f\in H(\Bbb C) : \, \displaystyle \lim_{r\rightarrow \infty} v(r)M_{p}(f,r) =0 \}.
$$
These spaces are Banach spaces with the norm
$$
\|f\|_{p,v} = \|f\|_{p,\infty,v}:= \sup_{r>0} v(r)M_{p}(f,r).
$$
According to \cite[Theorem 2.1]{lusky2000on}, the polynomials are contained in $B_{p,0}$ for all $1 \leq p \leq \infty$ and form a dense subset in it. In particular, each space \,$B_{p,0}$ \,is separable.

\vskip .15cm

It is worth noting that, if \,$X$ \,is a Banach space, then an operator \,$T \in B(X)$ \,happens to be distributionally chaotic in the sense of
Schweizer and J.~Sm\'{\i}tal \cite{schweizersmital} if and only if \,$T$ \,has a distributionally irregular vector \cite[Theorem 12]{BBMP}.

\begin{theorem} \label{theorem-weight}
{\it Let \,$v$ \,be a weight function such that \,$\lim _{r\rightarrow \infty}v(r)\frac {e^{r}}{r^{\frac{1}{2p}}} = 0$ \,for some $1\leq p \leq \infty$. If the differentiation
operator $D:B_{p,0} \rightarrow B_{p,0}$ is continuous, then there is a dense vector subspace of \,$B_{p,0}$ \,all of whose nonzero functions are distributionally irregular for this operator.}
\end{theorem}

\begin{proof}
By \cite[Theorem 2.3]{BB}, if  \,$v$ \,is a weight function such that $\lim _{r\rightarrow \infty} v(r)\frac {e^{r}}{r^{\frac{1}{2p}} }=0$ for some $1\leq p \leq \infty$ and \,$D:B_{p,0} \rightarrow B_{p,0}$ \,is continuous, then \,$D$ \,is frequently hypercyclic.  Moreover, the set $X_{0}$ of the polynomials is  a dense subset  in $B_{p,0}$ such that \,$D^{n}f$ \,tends to \,$0$ \,in \,$B_{p,0}$ \,for all \,$f \in X_{0}$.

\vskip .15cm

On the other hand, Bayart and Ruzsa have recently proved \cite[Corollary 15]{BR} that if $X$ is a Banach space and $T \in B(X)$ is a frequently hypercyclic operator such that \,$T^n$ \,tends pointwise to \,$0$ \,on a dense subset of \,$X$ \,then \,$T$ \,is distributionally chaotic. Hence \,$D$ \,is distributionally chaotic on $B_{p,0}$. And since, once again, there exists a dense set $X_0$ such that \,$D^nf$ \,tends pointwise to \,$0$ for all $f\in X_0$, then by  \cite[Theorem 15]{BBMP} \,$D$ \,admits a dense vector subspace consisting (except for zero) of distributionally irregular functions in $B_{p,0}$.
\end{proof}

\begin{remark}
For any $b \in \R$, let be the weight \,$v(r) := r^b \, e^{-r}$ \, for $r\ge b$, which is non-increasing and satisfies that \,$\sup_{r > 0} {v(r) \over v(r+1)} < \infty$. Then, according to \cite[Proposition 2.1]{BB}, the operator \,$D:B_{p,0} \rightarrow B_{p,0}$ \,is continuous.
As a consequence of Theorem \ref{theorem-weight} we obtain that for every \,$p \in [1,2)$
\,and every \,$\ve > 0$ \,there exists an entire function \,$f$ \,such that
$$
M_{p}(f,r)\leq C \displaystyle \frac {e^r}{r^{\frac{1}{2p}-\ve}} \quad (r > 0)
$$
for some constant \,$C>0$ \,as well as two increasing sequences of
integers \,$A=(n_k)_k$ \,and \,$B=(m_k)_k$ \,such that
\,$\udens(A) = \udens(B) = 1$, satisfying  that \,$\lim_{k \to \infty}  D^{n_k}f = 0$ in $B_{p,0}$ (hence $D^{n_k}f \to 0$ in the topology of $H(\C )$, because convergence in $B_{p,0}$ implies uniform convergence on compacta) \,and \,$\lim_{k \to \infty}
\|D^{m_k}f\|_{p,v} =\infty$. But we do not know whether there exists \,$m \in \N$ such that
\,$\lim_{k \to \infty} \|D^{m_k}f\|_{\infty, \overline {B(0,m)}} = \infty$; that is, we do not know whether such an \,$f$ \,is $D$-irregular
in \,$H(\C )$, so we have not obtained an improvement of Theorem \ref{main theorem}.
\end{remark}

\vskip 5pt

\noindent {\bf Acknowledgements.} The first author is partially supported by the Plan Andaluz de Investigación de la Junta de Andalucía FQM-127 Grant P08-FQM-03543 and by MEC Grant MTM2012-34847-C02-01. The second author is partially supported by MEC and FEDER, project no.~MTM2014-52376-P.

\medskip

{\scriptsize
$$
\begin{array}{lr}
\mbox{Luis Bernal-Gonz\'alez } & \mbox{ Antonio Bonilla } \\
\mbox{Departamento de An\'alisis Matem\'atico } & \mbox{ Departamento de An\'alisis Matem\'atico } \\
\mbox{Universidad de Sevilla } & \mbox{ Universidad de La Laguna } \\
\mbox{Facultad de Matem\'aticas, Apdo.~1160 } &  \mbox{ C/Astrof\'{\i}sico Francisco S\'anchez, s/n } \\
\mbox{Avda.~Reina Mercedes, 41080 Sevilla, Spain} & \mbox{ 38271 La Laguna, Tenerife, Spain } \\  \mbox{E-mail: {\tt lbernal@us.es} } & \mbox{ E-mail: {\tt abonilla@ull.es} }
\end{array}
$$}

\end{document}